\documentclass{article}
\usepackage{amsmath}
\usepackage{amsthm}
\usepackage{cite}
\usepackage{amscd}
\usepackage{authblk}
\usepackage[all]{xy}
\theoremstyle{definition}
\newtheorem{thm}{Theorem}[section]

\newtheorem{prop}[thm]{Proposition}
\newtheorem{lem}[thm]{Lemma}
\newtheorem{rem}[thm]{Remark}

\newtheorem*{ack}{Acknowledgement}
\numberwithin{equation}{section}
\usepackage{amssymb}
\usepackage[margin=30truemm]{geometry}

\newcommand{\Hom}{{\rm Hom}}
\newcommand{\Coh}{{\rm Coh}}
\newcommand{\Rep}{{\rm Rep}_k}

\newcommand{\Mod}{{\rm Mod}}
\newcommand{\Qcoh}{{\rm Qcoh}}
\newcommand{\Aut}{{\rm Aut}}
\newcommand{\Dim}{{\rm dim}}

\newcommand{\uni}{{\rm uni}}

\newcommand{\Gal}{{\rm Gal}}
\newcommand{\calO}{\mathcal{O}}
\newcommand{\calC}{\mathcal{C}}
\newcommand{\varx}{\overline{x}}
\newcommand{\Def}{\overset{{\rm def}}{=}}
\newcommand{\ab}{{\rm ab}}
\newcommand{\llangle}{\langle\langle}
\newcommand{\rrangle}{\rangle\rangle}
\newcommand{\Lie}{{\rm Lie}}
\newcommand{\frakm}{\mathfrak{m}}

\title{Semifinite Bundles and the Chevalley-Weil Formula}
\author{Shusuke Otabe
\footnote{Mathematical Institute, Tohoku University, 6-3 Aramakiaza, Aoba, Sendai, Miyagi 980-8578, Japan;
E-mail: {sb4m06@math.tohoku.ac.jp}
}}
\date{\vspace{-12mm}}

\begin{document}

\maketitle

\begin{abstract}
In our previous paper, we studied the category of semifinite bundles on a proper variety defined over a field of characteristic 0. As a result, we obtained the fact that for a smooth projective curve defined over an algebraically closed field of characteristic 0 with genus $g>1$,  Nori fundamental group acts faithfully on the unipotent fundamental group of its universal covering. However, it was not mentioned about any explicit module structure. In this paper, we prove that the Chevalley-Weil formula gives a description of it.
\end{abstract}

\thispagestyle{empty}

\section{Introduction}

In \cite{we38}, Weil proved that if a vector bundle $E$ on a compact Riemann surface $X$ can be trivialized by some finite unramified covering $Y\to X$, then there exist two different nonzero polynomials $f\neq g\in\mathbb{N}[t]$ satisfying $f(E)\simeq g(E)$.  Vector bundles having the latter property are called \textit{Weil-finite bundles}, or simply \textit{finite bundles}. In \cite{no76}, Nori proved that the converse is also true, i.e., that any finite bundle can be trivialized by some finite unramified covering. In fact, he established this correspondence in more general setting. Now we recall his idea in the case where $X$ is a proper smooth connected scheme over an algebraically closed field $k$ of characteristic $0$. Nori studied the category $\calC^N(X)$ of finite bundles on $X$ and he proved that it is a neutral Tannakian category over $k$ together with a neutral fiber functor $\omega_x:E\mapsto x^*E$  induced by a rational point $x\in X(k)$~\cite{dm82}. By Tannaka duality, there exists an affine group scheme $\pi_1^N(X,x)$ and the fiber functor $\omega_x$ induces an equivalence of tensor categories $\omega_x:\calC^N(X)\xrightarrow{\simeq}\Rep(\pi_1^N(X,x))$. The affine group scheme $\pi_1^N(X,x)$ is called \textit{the fundamental group scheme}, or \textit{Nori fundamental group}. Since each object in $\calC^N(X)$ is finite, it turns out that $\pi_1^N(X,x)$ is a profinite group scheme over $k$. Nori also gives a description of the inverse functor $\omega_x^{-1}$. Namely, there exists a pointed $\pi_1^N(X,x)$-torsor $(X_x^N,x^N)$ over $(X,x)$ such that
\begin{equation*}
\omega_x^{-1}(V,\rho)\simeq(X^N_x\times (V,\rho))/\pi_1^N(X,x).
\end{equation*}
The triple $(X^N_x,\pi_1^N(X,x),x^N)$ is a universal object in the category of all finite pointed torsors $(P,G,p)$ over $(X,x)$. Namely, for each finite group scheme $G$, the map
\begin{equation*}
\Hom(\pi_1^N(X,x),G)\to{{\rm Tors}}((X,x),G);\,\phi\mapsto ((X^N_x,x^N)\times (G,\phi))/\pi_1^N(X,x)
\end{equation*}
is bijective. Here ${{\rm Tors}}((X,x),G)$ stands for the set of isomorphism classes of pointed $G$-torsors over $(X,x)$. However, since now $k$ is an algebraically closed field of characteristic 0, any finite $G$-torsor having surjective  corresponding morphism $\pi_1^N(X,x)\to G$ is nothing but a finite connected Galois \'etale covering with group $G(k)$. Hence in this case, Nori fundamental group $\pi_1^N(X,x)$ is none other than the pro-constant group scheme associated to Grothendieck's geometric \'etale fundamental group $\pi_1(X,\varx)$~\cite{gr71}. In particular,  a vector bundle $E$ on $X$ is finite if and only if it is trivialized by some finite \'etale covering of $X$.

In \cite{no82}, Nori also studied a \textit{unipotent} variant of the above argument. Now we recall this. A vector bundle $E$ is said to be \textit{unipotent} if there exists a filtration
\begin{equation}\label{eq:filt}
E=E^0\supseteq E^1\supseteq\dots\supseteq E^n=0,
\end{equation}
in $\Coh(X)$ such that $E^{i}/E^{i+1}\simeq\calO_X$ for any $i$. Let $\calC^{\uni}(X,x)$ be the category of all unipotent bundles on $X$. Nori proved that for a proper geometrically connected and reduced scheme $X$ over a field $k$, the category $\calC^{\uni}(X)$ is a neutral Tannakian category over $k$ together with a neutral fiber functor $\omega_x=x^*$ and the corresponding affine group scheme $\pi_1^{\uni}(X,x)\Def\pi_1(\calC^{\uni}(X),\omega_x)$ controls all pointed unipotent torsors over $(X,x)$. 
We will call $\pi_1^{\uni}(X,x)$ the \textit{unipotent fundamental group}. In particular, in the case where $k$ is of characteristic 0, any unipotent affine algebraic group scheme over $k$ is uniquely determined by its Lie algebra, so we can see that the maximal abelian quotient $\pi_1^{\uni}(X,x)^{{\rm ab}}$ is uniquely determined by the vector space
\begin{equation*}
{{\rm Lie}}(\pi_1^{\uni}(X,x)^{{\rm ab}})\simeq\Hom(\pi_1^{\uni}(X,x),\mathbb{G}_a)^{\vee}\simeq H^1(X,\calO_X)^{\vee}.
\end{equation*}
Note that these isomorphisms are canonical. Furthermore, he showed that if $X$ is a curve with $\dim H^1(X,\calO_X)=g$, then there exists a non-canonical isomorphism
\begin{equation}\label{eq:freeness}
k[\pi_1^{\uni}(X,x)]^{\vee}\simeq\prod_{n=0}^{\infty}(H^1(X,\calO_X)^{\vee})^{\otimes n}\simeq k\llangle X_1,\dots X_g\rrangle.
\end{equation}

In \cite{ot15}, the author introduced the notion of \textit{semifinite bundles} on a proper geometrically-connected and reduced scheme $X$ over a field $k$ of characteristic 0. Here, a vector bundle $E$ on $X$ is said to be \textit{semifinite} if there exists a filtration as (\ref{eq:filt})
in the category $\Coh(X)$ such that each subquotient $E^{i}/E^{i+1}$ is finite. We denote by $\calC^{EN}(X)$ the category of semifinite bundles on $X$. If $X(k)\neq\emptyset$, again the category $\calC^{EN}(X)$ is a neutral Tannakian category over $k$ and each rational point $x\in X(k)$ gives a neutral fiber functor $\omega_x=x^*$~\cite[Proposition 2.14]{ot15}. Furthermore, from the definition, both $\calC^N(X)$ and $\calC^{\uni}(X)$ are full Tannakian subcategories of $\calC^{EN}(X)$ which are closed under taking subquotients in $\calC^{EN}(X)$. Therefore, the corresponding fundamental group  $\pi_1^{EN}(X,x)\overset{{\rm def}}{=}\pi_1(\calC^{EN}(X),x^*)$ has as a quotient both $\pi_1^N(X,x)$ and $\pi_1^{\uni}(X,x)$. In particular, we denote by $p^N$ the natural projection $\pi_1^{EN}(X,x)\twoheadrightarrow\pi_1^N(X,x)$. The kernel of $p^N$ was calculated in \cite[Section 4]{ot15}. Now we recall this. Let
\begin{equation}
(X^N_x,\pi_1^N(X,x),x^N)=\varprojlim_{S}(X_S,\pi(X,S,x),x_S)
\end{equation}
be the universal torsor associated to $(X,x)$~\cite[Section 2.2]{ot15}. Here $S$ runs over all finitely generated full Tannakian subcategories of $\calC^N(X)$. Note that if $k$ is an algebraically closed field of characteristic 0, then $(X^N_x,x^N)$ is none other than the universal \'etale covering of $(X,x)$. Now we have the following exact sequence:
\begin{equation}\label{eq:exact seq}
1\to\pi_1^{\uni}(X_x^N,x^N)\to\pi_1^{EN}(X,x)\xrightarrow{p^N}\pi_1^{N}(X,x)\to 1,
\end{equation}
where we write
\begin{equation}\label{eq:pi_1^E}
\pi_1^{\uni}(X_x^N,x^N)\overset{{\rm def}}{=}\varprojlim_S\pi_1^{\uni}(X_S,x_S).
\end{equation}
Furthermore, if $k$ is an algebraically closed field of characteristic 0, since $\pi_1^{\uni}(X_x^N,x^N)$ is pro-unipotent and $\pi_1^N(X,x)$ is pro-reductive, the above sequence always splits~\cite{hm69}. Now we fix a section $t$ of $p^N$. Then $t$ defines a conjugacy action of $\pi_1(X,\overline{x})=\pi_1^N(X,x)(k)$ on $\pi_1^{\uni}(X_x^N,x^N)$:
\begin{equation*}
\rho_t:\pi_1(X,\varx)\to\Aut(\pi_1^{\uni}(X_x^N,x^N))\simeq\Aut(k[\pi_1^{\uni}(X_x^N,x^N)]^{\vee}).
\end{equation*}
Let $X$ be a projective smooth curve over an algebraically closed field $k$ of characteristic 0 with genus $g>0$. The author shows that if $g=1$, then the above action $\rho_t$ is trivial, and that if $g\ge 2$, then the action $\rho_t$ is faithful~\cite[Theorem 4.12]{ot15}.
The aim of this paper is to determine the $\pi_1(X,\varx)$-module structure induced by $\rho_t$ on $k[\pi_1^{\uni}(X_x^N,x^N)]^{\vee}$. We claim the following:

\begin{thm}\label{thm:main}
If $X$ is a smooth projective curve defined over an algebraically closed field of characteristic 0 with genus $g>0$, then 
there exists a $\pi_1(X,\varx)$-module isomorphism:
\begin{equation*}
(k[\pi_1^{\uni}(X_x^N,x^N)]^{\vee},\rho_t)\simeq\prod_{n=0}^{\infty} \bigl(k\oplus (k[\pi_1(X,\varx)])^{\oplus g-1}\bigl)^{\otimes n}.
\end{equation*}
\end{thm}

\begin{rem}\label{rem:section}
In particular, Theorem \ref{thm:main} implies that the $\pi_1(X,\varx)$-module structure on $k[\pi_1^{\uni}(X_x^N,x^N)]^{\vee}$ induced by $\rho_t$ does not depend on the choice of the section $t$.
This fact is a consequence of the freeness of the unipotent fundamental group for a curve~(\ref{eq:freeness})~(cf.~Section 2). On the other hand, it seems that in the case where $\dim X\ge 2$, the Galois module structure on $k[\pi_1^{\uni}(X_x^N,x^N)]^{\vee}$ induced by $\rho_t$ depends on the choice of the section $t$.
\end{rem}

\begin{ack}
The author thanks Professor Takao Yamazaki for many clarifying discussions on the relation between his previous work and the Chevalley-Weil formula.
The author thanks Professor Takuya Yamauchi for suggesting many  examples of projective smooth higher dimensional varieties with infinitely abelian  fundamental group.  
The author is supported by JSPS, Grant-in-Aid for Scientific Research for JSPS fellows (16J02171).
\end{ack}


\section{The Chevalley-Weil Formula and Proof of Main Result}

Throughout this paper, $k$ always means an algebraically closed field of characteristic 0 and $X$ always means a projective smooth connected scheme defined over $k$ together with a fixed rational point $x\in X(k)$ and together with a fixed section $t$ of $p^N$~(cf.~(\ref{eq:exact seq})). 

We will deduce Theorem \ref{thm:main} from the \textit{Chevalley-Weil formula}~\cite{cw34}. To see the relation between them, we consider the abelianization of $\rho_t$:
\begin{equation}
\begin{aligned}
\rho^{{\rm ab}}:\pi_1(X,\varx)
&\xrightarrow{\rho_t}\Aut(\pi_1^{\uni}(X_x^N,x^N))\\
&\to\Aut(\pi_1^{\uni}(X_x^{N},x^N)^{{\rm ab}})
\simeq\Aut(\Lie(\pi_1^{\uni}(X_x^N,x^N)^{{\rm ab}})).
\end{aligned}
\end{equation}
Here, we write $\Lie(\pi_1^{\uni}(X_x^{N},x^N)^{\ab})\Def\varprojlim_S\Lie(\pi_1^{\uni}(X_S, x_S)^{\ab})$ and note that $\rho^{\ab}$ does not depend on the choice of the section $t$.
Hence, we obtain a $\pi_1(X,\varx)$-module
\begin{equation*}
\Lie(\pi_1^{\uni}(X_x^N,x^N)^{\ab})=(\Lie(\pi_1^{\uni}(X_x^{N}, x^N)^{\ab}),\rho^{{\rm ab}}).
\end{equation*}
Now let $X$ is a curve with genus $g>0$.
Then the following equation is the abelian version of Theorem \ref{thm:main}:
\begin{equation}\label{eq:main ab}
\Lie(\pi_1^{\uni}(X_x^N, x^N)^{\ab})\simeq k\oplus (k[\pi_1(X,\varx)])^{\oplus g-1}.
\end{equation}
However, $\rho^{{\rm ab}}$ is an inductive limit of subrepresentations:
\begin{equation}\label{eq:rho S ab}
\rho^{\ab}=\varinjlim_S\bigl(\Gal(X_S/X)\xrightarrow{\rho^{\ab}_S}
\Aut(\Lie(\pi_1^{\uni}(X_S, x_S)^{\ab}))\bigl)
\end{equation}
(cf. \cite[Proof of Theorem 4.12]{ot15}) where note that $\pi(X,S,x)(k)=\Gal(X_S/X)$, so for the proof of (\ref{eq:main ab}), it suffices to prove that for each $S$, the following isomorphism exists:
\begin{equation}\label{eq:main ab S}
\Lie(\pi_1^{\uni}(X_S, x_S)^{\ab})\simeq k\oplus (k[\Gal(X_S/X)])^{\oplus g-1}.
\end{equation}
On the other hand, since $\calO_{X_S}=\pi_S^*\calO_X$, where $\pi_S:X_S\to X$ is the natural projection, the cohomology  $H^1(X_S,\calO_{X_S})^{\vee}$ has a natural $\Gal(X_S/X)$-module structure. This Galois module structure is well-understood by the \textit{Chevalley-Weil formula}~\cite{cw34}:
\begin{equation}\label{eq:CW}
H^1(X_S,\pi_S^*\calO_X)^{\vee}\simeq k\oplus (k[\Gal(X_S/X)])^{\oplus g-1}.
\end{equation}
Therefore, for the proof of (\ref{eq:main ab S}), it suffices to show:

\begin{prop}\label{prop:main}
For any smooth proper connected scheme $X$ over an algebraically closed field $k$ of characteristic 0 and for each finitely generated Tannakian full subcategory $S$ of $\calC^N(X)$, there exists a $\Gal(X_S/X)$-module isomorphism:
\begin{equation*}
\Lie(\pi_1^{\uni}(X_S, x_S)^{\ab})\simeq H^1(X_S,\pi_S^*\calO_X)^{\vee}.
\end{equation*}
\end{prop}

In fact, in the case where $X$ is a curve over $k$,  Proposition \ref{prop:main} implies Theorem \ref{thm:main}:

\begin{proof}[Proof of Proposition \ref{prop:main} $\Rightarrow$ Theorem \ref{thm:main}]
Let $X$ be a projective smooth connected curve over $k$ of genus $g>0$. As (\ref{eq:rho S ab}), also $\rho$ can be written as an  inductive limit of subrepresentations:
\begin{equation}\label{eq:rho S}
\rho_t=\varinjlim_S\bigl(\Gal(X_S/X)
\xrightarrow{\rho_S}\Aut(A_S)\bigl),
\end{equation}
where we write $A_S\Def k[\pi_1^{\uni}(X_S,x_S)]^{\vee}$~(cf.~\cite[Proof of Theorem 4.12]{ot15}). Therefore, for Theorem \ref{thm:main}, it suffices to show that for each $S$, we have:
\begin{equation}\label{eq:A_S,EN}
A_{S}\simeq \prod_{n=0}^{\infty}\bigl(k\oplus(k[\Gal(X_S/X)])^{\oplus g-1}\bigl)^{\otimes n}.
\end{equation}
However, note that there exists a canonical isomorphism of vector spaces over $k$ as below:
\begin{equation}\label{eq:decomp}
A_{S}\simeq \varprojlim_n A_S/\frakm_S^n\simeq
\prod_{n=0}^{\infty}\mathfrak{m}_S^n/\mathfrak{m}_S^{n+1}.
\end{equation}
Here we denote by $\frakm_S$ the unique maximal ideal of $A_S$.
On the other hand, since $\rho_S$ gives $k$-algebra automorphisms of $A_S$ and since $\Dim X_S=1$, by the freeness of $A_S$~(cf.~(\ref{eq:freeness})), there exist  canonical isomorphisms of $\Gal(X_S/X)$-modules:
\begin{equation*}
\Lie(\pi_1^{\uni}(X_S,x_S)^{\ab})^{\otimes n}\xrightarrow{\simeq}(\frakm_S/\frakm^2_S)^{\otimes n}\xrightarrow{\simeq}\frakm_S^n/\frakm_S^{n+1}\, (n>0)
\end{equation*}
and the action of $\Gal(X_S/X)$ on $A_S$ preserves the decomposition (\ref{eq:decomp}).
Therefore, Proposition \ref{prop:main} and the Chevalley-Weil formula (\ref{eq:CW}) imply (\ref{eq:A_S,EN}), whence Theorem \ref{thm:main}.
\end{proof}

It remains to show Proposition \ref{prop:main}. 
Let $X/k$ be a smooth proper connected scheme. 

\begin{lem}\label{lem:equivalence}
(1) If $\pi:P\to X$ is an affine morphism, then we have an equivalence of categories
\begin{equation*}
\Qcoh(P)\simeq \Mod_X(\pi_*\calO_P),
\end{equation*}
where $\Mod_X(\pi_*\calO_P)$ stands for the category of $\pi_*\calO_P$-module objects in $\Qcoh(X)$.

(2) Furthermore, if $\pi:P\to X$ is an \'etale Galois covering with group $G$, then under the identification $G=\Aut(P/X)\simeq\Aut_{\text{$\calO_X$-alg}}(\pi_*\calO_P)^{{\rm op}},\,\sigma\leftrightarrow\sigma^{-1}$, for each $\sigma\in G$, the following diagram is commutative:
\begin{equation*}
\begin{CD}
\Qcoh(P)@>\simeq >>\Mod_X(\pi_*\calO_P)\\
@V\sigma^* VV @VV{{\rm Res}}(\sigma^{-1})V\\
\Qcoh(P)@>\simeq >>\Mod_X(\pi_*\calO_P).
\end{CD}
\end{equation*}
\end{lem}

\begin{proof}
(1) It suffices to show that for any $k$-algebra homomorphism $A\to B$, there exists an equivalence of categories:
$\Mod(B)\simeq \Mod_A(B_A)$. Here for each $M\in\Mod(B)$, we write $M_A\Def{{\rm Res}}(\phi)(M)\in\Mod(A)$. Similar for each morphism in $\Mod(B)$. For each $M\in\Mod(B)$, let $\mu_M:B\otimes_k M\to M$ be the structure homomorphism. Then the associativity of $\mu_M$ implies that $\mu_M$ is a $B$-linear homomorphism. Thus we obtain a $A$-linear homomorphism
$(\mu_M)_A:(B\otimes_k M)_A\to M_A$.
This factors uniquely through the natural surjection $(B\otimes_k M)_A\twoheadrightarrow B_A\otimes_A M_A$ and the induced $A$-linear homomorphism $\mu_{M_A}:B_A\otimes_A M_A\to M_A$ makes $M_A$ a $B$-module object in $\Mod(A)$. We define a functor $\Phi:\Mod(B)\to\Mod_A(B_A)$ by $\Phi((M,\mu_M))\Def (M_A,\mu_{M_A})$. On the other hand, let $N\in\Mod_A(B_A)$ and assume that the $A$-linear homomorphism  $\mu_N:B_A\otimes_A N\to N$ gives a $B_A$-module structure on $N$. Then the composition
$B\otimes_k N\twoheadrightarrow B_A\otimes_A N\xrightarrow{\mu_N}N$ makes $N$ a $B$-module, hence we obtain a functor $\Psi:\Mod_A(B_A)\to\Mod(B)$. By construction, we have $\Phi\circ\Psi={{\rm id}}$ and $\Psi\circ\Phi={{\rm id}}$. Therefore $\Phi$ is an equivalence of categories.

(2) This follows from the fact that $\pi_*\sigma^*E\simeq {{\rm Res}(\sigma^{-1})}(\pi_*E)$ in $\Mod_X(\pi_*\calO_P)$ for each $E\in\Qcoh(P)$.
\end{proof}

\begin{proof}[Proof of Proposition \ref{prop:main}]
Let $S\subset\calC^N(X)$ be a finitely generated Tannakian full subcategory of $\calC^N(X)$ and $\overline{S}\subset\calC^{EN}(X)$ the subcategory obtained by successive extensions of finite bundles in $S$. To ease of notation, we will simply write $H,G$ and $U$ for $\pi(X,S,x),\pi(X,\overline{S},x)$ and $\pi_1^{\uni}(X_S,x_S)$, respectively. By Lemma \ref{lem:equivalence}(1), we obtain the following commutative diagram~(cf.~\cite[Remark 4.8]{ot15}):
\begin{equation*}
\begin{xy}
\xymatrix{
              && \Mod_{\overline{S}}({\pi_S}_*\calO_{X_S})\\
S\ar @{^{(}-{>}} [r]\ar[d]^{\simeq}&  \overline{S}\ar[r]_{{\pi_S}^*\quad}\ar[d]^{\simeq}\ar[ru]^{\alpha'}&  \calC^{\uni}(X_S)\ar[u]_{\simeq}\ar[d]^{\simeq}\\
\Rep(H)\ar @{^{(}-{>}} [r]&  \Rep(G)\ar[r]\ar[rd]_{\alpha}&  \Rep(U)\ar[d]^{\simeq}\\
                        &                                      &  \Mod_{\Rep(G)}(P).
}
\end{xy}
\end{equation*}
Here $\Mod_{\overline{S}}({\pi_S}_*\calO_{X_S})$ stands for the category of ${\pi_S}_*\calO_{X_S}$-module objects in $\overline{S}$. Similar for $\Mod_{\Rep(G)}(P)$, where $P\Def(k[H],\rho_{{\rm reg}})\in\Rep(H)$, which corresponds to ${\pi_S}_*\calO_{X_S}\in S$ via the equivalence $\Rep(H)\simeq S$. Since for each $E\in \overline{S}$, we have $\alpha'(E')={\pi_S}_*\pi_S^*E\simeq{\pi_S}_*\calO_{X_S}\otimes E$, the functor $\alpha$ is given by $\alpha(V)=P\otimes V$.
Now let $G=H\ltimes U$ and fix $\sigma\in H(k)=\Gal(X_S/X)$. Then the conjugation  action $c_\sigma$ on $G$ induces the action $\Phi_\sigma$ on $\Mod_{\Rep(G)}(P)$ which makes the following diagram commute:
\begin{equation*}
\begin{CD}
\Rep(H)@>>>\Rep(G)@>\alpha>>\Mod_{\Rep(G)}(P)\\
@V{{\rm Res}}(c_\sigma)VV @V{{\rm Res}}(c_\sigma)VV @VV\Phi_{\sigma}V\\
\Rep(H)@>>>\Rep(G)@>\alpha>>\Mod_{\Rep(G)}(P).
\end{CD}
\end{equation*}
We claim that for each $((V,\rho_V),\mu_V:P\otimes V\to  V)\in\Mod_{\Rep(G)}(P)$, we have
\begin{equation}\label{eq:Phi is Res}
\Phi_\sigma((V,\rho_V),\mu_V)=({{\rm Res}}(c_\sigma)(V,\rho_V),{{\rm Res}}(\sigma^{-1})(\mu_V)).
\end{equation}
First note that there exists an isomorphism of algebra objects in $\Rep(G)$:
\begin{equation*}
P_{\sigma}\Def{{\rm Res}}(c_{\sigma})(P)\xleftarrow[\simeq]{\sigma^{-1}}P
\end{equation*}
and the functor ${{\rm Res}}(c_{\sigma})$ induces the trivial action on Hom-sets in ${\Rep(G)}$. Now the composition
\begin{equation}\label{eq:def of Phi_sigma}
\begin{aligned}
P\otimes({{\rm Res}}(c_\sigma)(V,\rho_V))&\xrightarrow[\simeq]{\sigma^{-1}\otimes{{\rm id}}}
P_{\sigma}\otimes{{\rm Res}}(c_\sigma)(V,\rho_V)\\
&={{\rm Res}}(c_{\sigma})(P\otimes(V,\sigma_V))\xrightarrow{{{\rm Res}}(c_\sigma)(\mu_V)=\mu_V}{{\rm Res}}(c_\sigma)(V,\rho_V)
\end{aligned}
\end{equation}
gives a $P$-module structure on $\Phi_\sigma((V,\rho_V),\mu_V)$, but the morphism (\ref{eq:def of Phi_sigma}) is nothing but ${{\rm Res}}(\sigma^{-1})(\mu_V)$, whence the equality (\ref{eq:Phi is Res}). Therefore, by Lemma \ref{lem:equivalence}(2), the diagram
\begin{equation}\label{eq:Phi sigma*}
\begin{CD}
\Mod_{\Rep(G)}(P)@>\simeq>>\calC^{\uni}(X_S)\\
@V\Phi_\sigma VV @VV\sigma^*V\\
\Mod_{\Rep(G)}(P)@>\simeq>>\calC^{\uni}(X_S)
\end{CD}
\end{equation}
commutes. Since the action of $\sigma\in\Gal(X_S/X)$ on $H^1(X_S,\pi_S^*\calO_X)$ is induced by the pull-back functor $\sigma^*$, the commutativity of the diagram (\ref{eq:Phi sigma*}) implies  Proposition \ref{prop:main}.
\end{proof}

This completes the proof of Theorem \ref{thm:main}.


\section{Remarks}

\begin{rem}
If $X$ is an abelian variety over $k$ of dimension $g>0$ and $n>0$ is an integer, then the multiplication map $n:X\to X$ gives a finite Galois \'etale covering of $X$ with group $G\simeq(\mathbb{Z}/n\mathbb{Z})^{\oplus 2g}$. Each element $\sigma\in G$ acts on $X$ by a parallel transformation,  hence the action of $\sigma$ on $H^1(X,\calO_X)$ is trivial because for any unipotent bundle $E$ on an abelian variety and for any point $a\in X(k)$, we have $T_a^*E\simeq E$~\cite[Theorem 4.17]{mu78}, where $T_a:X\to X$ is the parallel transformation defined by $p\mapsto p+a$.

This fact on an abelian variety can be extended to more general setting. Namely, if $X$ is a projective smooth connected scheme defined over an algebraically closed field of characteristic 0  with $\pi_1(X,\varx)$ abelian, then for any finite Galois \'etale covering $\pi:P\to X$ with group $G$, the $G$-module structure on $H^1(P,\pi^*\calO_X)$ is trivial. 
Indeed, by the Lefschetz principle, we may assume that $k=\mathbb{C}$.  By Malcev-Grothendieck's theorem~\cite{ma40}\cite{gr70}, the commutativity of $\pi_1(X,\varx)$ implies that the topological fundamental group $\pi_1^{{\rm top}}(X(\mathbb{C}),x)$ is commutative. Then by \cite[Corollary 1.5]{la11}, the $S$-fundamental group $\pi_1^S(X,x)$ is also abelian, hence so is $\pi_1^{EN}(X,x)$. (Note that $\pi_1^{EN}(X,x)$ is a quotient of $\pi_1^S(X,x)$.) Thus, by \cite[9.5 Theorem]{wa79}, we obtain $\pi_1^{EN}(X,x)\simeq\pi_1^{N}(X,x)\times\pi_1^{\uni}(X,x)$. This implies that $\pi_1^{\uni}(X_x^N,x^N)\xrightarrow{\simeq}\pi_1^{\uni}(X,x)$, whence the equivalence $\pi^*:\calC^{\uni}(X)\xrightarrow{\simeq}\calC^{\uni}(P)$~(in particular, $\pi^*:H^1(X,\calO_X)\xrightarrow{\simeq} H^1(P,\calO_P)$), and that the action $\rho_t:\pi_1(X,\varx)\to\Aut(k[\pi_1^{\uni}(X,x)]^{\vee})$ is trivial. By  Proposition \ref{prop:main}, the triviality of $\rho_t$ deduces the fact that the $G$-module structure on $H^1(P,\pi^*\calO_X)$ is trivial.
\end{rem}

\begin{rem}
Let $X$ be a projective smooth connected curve over an algebraically closed field $k$ of characteristic 0 with genus $g\ge 2$. Let $\pi:P\to X$ be a finite Galois \'etale covering with group $G$. We have seen that the conjugation action $\rho_t$ endows each $\frakm_P^n/\frakm_P^{n+1}$ with the structure of $G$-module and that if $n=1$, then this $G$-module structure coincides with the natural one on $H^1(P,\pi^*\calO_X)^{\vee}$. In fact this observation can be valid also for $n\ge 2$. Indeed, we claim that there exist $G$-module isomorphisms
\begin{equation}\label{eq:V_n}
\frakm_P^n/\frakm_P^{n+1}\simeq H^1(P,\pi^*(V_n\otimes\Omega_{X/k}^1)^{\vee})^{\vee}\quad(n>0).
\end{equation}
Here, the sheaves $V_n\in\Coh(X)$ are defined by the following way. For each integer $n>0$, let $U_n$ be the \textit{n-th universal extension} in $\calC^{\uni}(P)$ with $U_1=\calO_P$~\cite[Chapter IV]{no82}. Then the Galois group $G$ acts on the sheaf $\Omega_{P/k}^1\otimes U_n=\pi^*\Omega_{X/k}^1\otimes U_n$ via the action on $\pi^*\Omega_{X/k}^1$, so we obtain a sheaf $V_n\Def (\Omega^1_{P/k}\otimes U_n)^G$ over $X$. Note that all the $V_n$ are locally free sheaves. Now we prove (\ref{eq:V_n}). 
By Theorem \ref{thm:main}, we have
\begin{equation*}
\frakm_P^n/\frakm_P^{n+1}\simeq (\frakm_P/\frakm_P^2)^{\otimes n}\simeq \bigl(k\oplus (k[G]^{\oplus g-1})\bigl)^{\otimes n}\quad (n>0),
\end{equation*}
so it suffices to show that
\begin{equation}\label{eq:CW V_n}
H^1(P,\pi^*(V_n\otimes\Omega_{X/k}^1)^{\vee})\simeq \bigl(k\oplus (k[G]^{\oplus g-1})\bigl)^{\otimes n}\quad (n>0).
\end{equation}
From the Serre duality and a basic property of the universal extensions, we find that
\begin{equation*}
\begin{aligned}
H^0(P,\pi^*(V_n\otimes\Omega_{X/k}^1)^{\vee})&\simeq H^1(P,\pi^*V_n)^{\vee}=H^1(P,\Omega_{P/k}^1\otimes U_n)^{\vee}\\
&\simeq H^0(P,U_n^{\vee})\simeq H^0(P,\calO_P)=k
\end{aligned}
\end{equation*}
is a trivial $G$-module.
So we can adopt the proof of \cite[Theorem 4]{na84}. Then, by the semi-simplicity of $k[G]$, we find that there exists a $G$-module isomorphism
\begin{equation}\label{eq:CW V_n 2}
H^1(P,\pi^*(V_n\otimes \Omega^1_{X/k})^{\vee})\simeq k\oplus (k[G])^{\oplus  (g_P^n-1)/n},
\end{equation}
where $g_P\Def\dim H^1(P,\calO_P)$.
By the Hurwitz genus formula, we can find that (\ref{eq:CW V_n 2}) implies (\ref{eq:CW V_n}). Therefore, we obtain the isomorphisms (\ref{eq:V_n}).
\end{rem}


\end{document}